\documentclass[draft]{amsart}

\usepackage{pgf,tikz}
\usetikzlibrary{arrows}

\usepackage{amssymb,graphicx,enumerate,color}
\usepackage{multirow}

\newcommand{\C}{\tilde{C}}

\newcommand{\Sum}{\displaystyle\sum}

\renewcommand{\C}{\mathcal{C}}
\renewcommand{\P}{\mathcal{P}}
\newcommand{\U}{\mathcal{U}}
\newcommand{\D}{{\mathcal{D}}}

\newcommand{\R}{{\mathcal{R}}}
\newcommand{\M}{\mathcal{M}}

\newcommand{\x}{\xi}

\newcommand{\E}{{\mathcal{E}}}

\newcommand{\K}{{\mathcal{K}}}

\newcommand{\beqs}{\begin{equation*}}
\newcommand{\eeqs}{\end{equation*}}
\numberwithin{equation}{section}
 \theoremstyle{plain}
\newtheorem{theorem}{Theorem}[section]

\newtheorem{corollary}[theorem]{Corollary}

\theoremstyle{remark}

\begin{document}
\makeatletter
\def\imod#1{\allowbreak\mkern10mu({\operator@font mod}\,\,#1)}
\makeatother

\author{Ali Kemal Uncu}
   \address{Department of Mathematics, University of Florida, 358 Little Hall, Gainesville FL 32611, USA}
   \email{akuncu@ufl.edu}

\title[\scalebox{.9}{Weighted Rogers--Ramanujan Partitions and Dyson Crank}]{Weighted Rogers--Ramanujan Partitions and Dyson Crank}
     
\begin{abstract} In this paper we refine a weighted partition identity of Alladi. We write formulas for generating functions for the number of partitions grouped with respect to a partition statistic other than the norm. We tie our weighted results and the different statistic with the crank of a partition. In particular, we prove that the number of partitions into even number of distinct parts whose odd-indexed parts' sum is $n$ is equal to the number of partitions of $n$ with non-negative crank.
\end{abstract}   
   
  


\date{\today}

\maketitle

\section{Introduction and Notations} 
A \textit{partition} is a finite sequence $\pi=(\lambda_1,\lambda_2,\dots)$ of decreasing (not necessarily strict) positive integers. The elements of the sequence $\pi$ are called the \textit{parts} of the partition $\pi$. We define the \textit{norm} of a partition $\pi$ as the sum of all its parts, $\lambda_1+\lambda_2+\dots$, and this will be denoted as $|\pi|$. As an example, there are 5 partitions, $(4),\ (3,1),\ (2,2),\ (2,1,1),\ (1,1,1,1)$, with norm equal to 4. For an integer $n$, we will use \textit{partitions of $n$} to denote the set of all the partitions with norm $n$. Abiding the general convention, we accept the empty sequence as a partition, and it is the unique partition of 0. 

The norm of partitions is one of the the most natural statistics. There are finitely many partitions with a fixed norm. This makes the norm a great candidate for indexing generating functions. The theory of partitions is primarily concerned with the relationship between the sizes of different sets of partitions where elements from both sets have the same norm. One early example is due to Euler \cite{Theory of Partitions}.

\begin{theorem}[Euler, 1748]\label{EulerTHM} The number of partitions of $n$ into distinct parts is the same as the number of partitions of $n$ into odd parts.
\end{theorem}

Theorem~\ref{EulerTHM} and many other theorems of the same spirit utilizes generating functions in their proofs. Let $A$ be a set of partitions, and let $p_A(n)$ be the number of partitions in $A$ with norm $n$. Then \begin{equation}\label{GENFUNC}\sum_{\pi\in A} 1 \cdot q^{|\pi|} = \sum_{n\geq 0} p_A(n) q^n\end{equation} is the generating function for the number of partitions with the same norm from the set $A$ written in two separate combinatorial ways, abstract and enumerative respectively. Here it is clear that every partition $\pi \in A$ makes a contribution of one to the $q^{|\pi|}$ term.

We would like to introduce four classically studied sets of partitions. 
\begin{enumerate}[i.]\item Let $\U$ be the set of all (unrestricted) partitions. 
\item Let $\D$ be the set of all partitions into distinct parts. 
\item Let $\R\R_1$ be the set of all partitions with difference between parts $\geq 2$. 
\item Let $\R\R_2$ be the set of all partitions with difference between parts $\geq 2$ where parts are $>1$. \end{enumerate} These listed sets are nested: $\R\R_2\subset\R\R_1\subset \D\subset\U$. The generating functions for the number of partitions from these sets are extensively studied in the literature. 

One can generalize the classical approach of writing abstract generating functions with respect to the norm \eqref{GENFUNC} by attaching weights in the place of 1. In 1997,  Alladi \cite{AlladiWeighted} inquired about the existence and identification of a weight $\omega_S(\pi)$ on a set of partitions $S$ so that \begin{equation}\label{GF_function_abstract}\sum_{\pi\in S} \omega_S(\pi)q^{|\pi|} = \sum_{\pi\in T} q^{|\pi|}\end{equation} for some set of partitions $T$ that contains $S$. He proved the interesting result, which exemplifies the existence of solutions of \eqref{GF_function_abstract}:

\begin{theorem}[Alladi, 1997] \label{Alladi_weighted_sum} Let $\nu(\pi)$ denote the number of parts of $\pi$. Then \begin{equation}\label{omega_12}\sum_{\pi\in \R\R_1} \omega_{1,2}(\pi) q^{|\pi|} = \sum_{\pi\in U} q^{|\pi|}\end{equation}where
\[\omega_{1,2}(\pi) := \lambda_{\nu(\pi)}\cdot\prod_{i=1}^{\nu(\pi)-1} (\lambda_{i}-\lambda_{i+1}-1),\]
and weight of the empty sequence is considered to be the empty product, and is set equal to 1.
\end{theorem} Similar weighted identities and their interesting applications have been discussed \cite{AlladiWeighted}, \cite{AlladiBerkovich}, and \cite{AlladiBerkovich2}.

It should be noted that the relation $T\subset S$ in \eqref{GF_function_abstract} is of little interest. In this case one can define the weight $\omega_S(\pi)$ to be the indicator function \[\omega_S(\pi) := \left\{ \begin{array}{ll}
1, &\text{if }\pi\in T,\\
0,&\text{otherwise.}
\end{array} \right.\]

Our main motivation lies in the similar question to the one of Alladi's. We would like to identify statistics $\Lambda$ such that for sets of partitions $S\subset T$ we have \begin{equation}\label{GF_abstract_EXP_Weight}\sum_{\pi\in S}q^{\Lambda(\pi)}=\sum_{\pi\in T} q^{|\pi|}.\end{equation}Later we prove the following result:

\begin{theorem}\label{AliTHM1}
\[\sum_{\pi\in \D} q^{\mathcal{O}(\pi)} = \sum_{\pi \in \U} q^{|\pi|},\]
where $\mathcal{O}(\pi) := \lambda_1 + \lambda_3+\dots$, the sum of the odd indexed parts, for a partition $\pi=(\lambda_1,\lambda_2,\dots)$.
\end{theorem}

Similar to the problem of identifying weights, the case $T \subset S$ is trivial since one can formally pick \[\Lambda(\pi)=\left\{ \begin{array}{ll}
|\pi|, &\text{if }\pi\in T,\\
\infty, &\text{otherwise,}
\end{array} \right. \]
where we assume $|q|<1$.

For $i\in\{1,2\}$, identifying the weights $\omega_i(\pi)$, the partition statistics $\Lambda_i$, and sets of partitions (or vector partitions) $S$ and $T$ that satisfy
\begin{equation}
\label{Constant_Weights_Def}\sum_{\pi\in S}\omega_1(\pi) q^{\Lambda_1(\pi)} = \sum_{\pi\in T} \omega_2(\pi)q^{\Lambda_2(\pi)} \\
\end{equation} is an enveloping generalization of the mentioned questions related with \eqref{GF_function_abstract} and \eqref{GF_abstract_EXP_Weight}. This general question reduces to the classical combinatorial study of partition identities for $\omega_i(\pi)\equiv 1$ and $\Lambda_i(\pi)\equiv |\pi|$ with sets of partitions $S$ and $T$. One example of this particular case is Theorem~\ref{EulerTHM}. 

In Section~\ref{Section2} we define $q$-Pochhammer symbols, and the Ferrers diagrams. We also remark some well-known results for completeness of the paper. Section~\ref{Section3} has the refinement and a proof of Theorem~\ref{Alladi_weighted_sum}. The crank of a partition and its relation with both the weighted identities and different partition statistics is given in Section~\ref{Section4}. Section~\ref{Section5} is devoted for a short excursion of writing generating functions with respect to the partition statistics, sum of the odd-indexed parts of a partition.

\section{Some Basics of Partition Theory}\label{Section2}

The Ferrers diagram of a partition $\pi=(\lambda_1,\lambda_2,\dots)$ is a graphical representation of the parts of $\pi$, \cite{Theory of Partitions}, where we put $\lambda_i$ many dots in the integral coordinates on the $i$-th row from the top of the diagram to represent this part. Two examples of such representations are the Ferrers diagrams of $(4,4,2,1,1)$ and $(5,3,2,2)$ respectively:

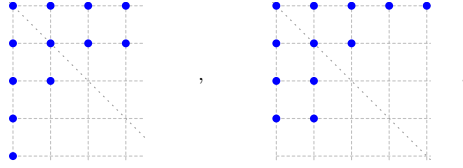
\begin{center}
\begin{tabular}{cc}
\definecolor{zzzzzz}{rgb}{0.6,0.6,0.6}
\definecolor{qqqqff}{rgb}{0,0,1}
\definecolor{cqcqcq}{rgb}{0.75,0.75,0.75}
\begin{tikzpicture}[line cap=round,line join=round,>=triangle 45,x=.5cm,y=.5cm]
\draw [color=cqcqcq,dash pattern=on 1pt off 1pt, xstep=0.5cm,ystep=.5cm] (1,-5.1) grid (4.5,-1);
\clip(0.9,-5.1) rectangle (6.5,-0.9);
\draw [dotted,color=zzzzzz,domain=0.9:4.5] plot(\x,{(-0-1*\x)/1});
\begin{scriptsize}
\fill [color=qqqqff] (1,-1) circle (1.5pt);
\fill [color=qqqqff] (2,-1) circle (1.5pt);
\fill [color=qqqqff] (3,-1) circle (1.5pt);
\fill [color=qqqqff] (4,-1) circle (1.5pt);
\fill [color=qqqqff] (1,-2) circle (1.5pt);
\fill [color=qqqqff] (2,-2) circle (1.5pt);
\fill [color=qqqqff] (3,-2) circle (1.5pt);
\fill [color=qqqqff] (4,-2) circle (1.5pt);
\fill [color=qqqqff] (1,-3) circle (1.5pt);
\fill [color=qqqqff] (2,-3) circle (1.5pt);
\fill [color=qqqqff] (1,-4) circle (1.5pt);
\fill [color=qqqqff] (1,-5) circle (1.5pt);
\draw (6,-3) node[anchor=center] {,};
\end{scriptsize}
\end{tikzpicture} &
\definecolor{zzzzzz}{rgb}{0.6,0.6,0.6}
\definecolor{qqqqff}{rgb}{0,0,1}
\definecolor{cqcqcq}{rgb}{0.75,0.75,0.75}
\begin{tikzpicture}[line cap=round,line join=round,>=triangle 45,x=0.5cm,y=0.5cm]
\draw [color=cqcqcq,dash pattern=on 1pt off 1pt, xstep=0.5cm,ystep=0.5cm] (1,-5.1) grid (5.1,-1);
\clip(0.9,-5.1) rectangle (6.1,-0.9);
\draw [dotted,color=zzzzzz,domain=0.9:5.1] plot(\x,{(-0-1*\x)/1});
\begin{scriptsize}
\fill [color=qqqqff] (1,-1) circle (1.5pt);
\fill [color=qqqqff] (2,-1) circle (1.5pt);
\fill [color=qqqqff] (3,-1) circle (1.5pt);
\fill [color=qqqqff] (4,-1) circle (1.5pt);
\fill [color=qqqqff] (1,-2) circle (1.5pt);
\fill [color=qqqqff] (2,-2) circle (1.5pt);
\fill [color=qqqqff] (3,-2) circle (1.5pt);
\fill [color=qqqqff] (1,-3) circle (1.5pt);
\fill [color=qqqqff] (2,-3) circle (1.5pt);
\fill [color=qqqqff] (1,-4) circle (1.5pt);
\fill [color=qqqqff] (2,-4) circle (1.5pt);
\fill [color=qqqqff] (5,-1) circle (1.5pt);
\end{scriptsize}
\draw (6,-3) node[anchor=center] {.};
\end{tikzpicture}
\end{tabular}
\end{center}

We note that taking the symmetric images of points in the Ferrers diagram over the (main) diagonal line gives us a Ferrers diagram of a partition. Two partitions whose Ferrers diagrams that are related by symmetry over the main diagonal are said to be \textit{conjugate} of each other. In our example  $(4,4,2,1,1)$ and $(5,3,2,2)$ are conjugate partitions. 

One should also stress that the partitions can be identified by their Ferrers diagrams and vice versa. From now on we will be using the notation $\pi$ for a partition or a partition's Ferrers diagram interchangeably.  

For the product representations of generating functions of interest we define the $q$-Pochhammer symbols \cite{Theory of Partitions}. Let $L$, and $k$ be non-negative integers. The $q$-Pochhammer symbol is
\begin{align*}
	(a)_L:=(a;q)_L &:= \prod_{n=0}^{L-1}(1-aq^n).\\
	\intertext{Some abbreviations of the notation we are going to use are}
	(a_1,a_2,\dots,a_k;q)_L &:= (a_1;q)_L(a_2;q)_L\dots (a_k;q)_L,\\
	(a;q)_\infty &:= \lim_{L\rightarrow\infty} (a;q)_L,\text{ where }|q|<1.
\end{align*}

With these definitions we can write explicit formulas of the generating functions on the defined sets in multiple ways. 

\begin{align}
\label{GF_unrestricted_pts}\sum_{\pi\in \U} q^{|\pi|} &=\sum_{n\geq 0} \frac{q^{n^2}}{(q)^2_n}=\frac{1}{(q)_\infty},\\
\label{GF_distinct_pts}\sum_{\pi\in \D} q^{|\pi|} &=(-q)_\infty = \frac{1}{(q;q^2)_\infty},\\
\label{RR1_identity}\sum_{\pi\in \R\R_1} q^{|\pi|} &=  \sum_{n\geq0} \frac{q^{n^2}}{(q)_n} = \frac{1}{(q,q^4;q^5)_\infty},\\
\label{RR2_identity}\sum_{\pi\in \R\R_2} q^{|\pi|} &= \sum_{n\geq0} \frac{q^{n^2+n}}{(q)_n} = \frac{1}{(q^2,q^3;q^5)_\infty}.
\end{align}

Second equality of \eqref{GF_distinct_pts} is the analytic version of Theorem~\ref{EulerTHM}. The extreme right equality of \eqref{RR1_identity} and \eqref{RR2_identity} are the well-celebrated Rogers--Ramanujan identities. Let $\widehat{C}_1$ (and $\widehat{C}_2$) be the set of all the partitions into parts $\equiv \pm1\mod{5}$ (and $\equiv \pm 2\mod{5}$). Hence, for $i\in\{1,2\}$
\begin{equation}
\label{RR_combinatorics_example} \sum_{\pi\in \R\R_i} q^{|\pi|} = \sum_{\pi\in \widehat{C}_i} q^{|\pi|}. \\
\end{equation}
This is a classical example of \eqref{Constant_Weights_Def}. Moreover, \eqref{RR_combinatorics_example} can be equivalently stated in its enumerative form \cite{Theory of Partitions}:
\begin{theorem}[Rogers-Ramanujan, 1919] \label{RR_combinatorial_THM}Let $n$ be any non-negative integer. The number of partitions of $n$ into parts with minimal distance of 2 between consecutive parts (where the smallest part $\not=1$) is equal to the number of partitions of $n$ into parts $\equiv \pm1\mod{5}$ (and $\equiv \pm 2\mod{5}$).
\end{theorem}

\section{Partition Identities Involving Weights}\label{Section3}

Let $M$, $k-1$, and $m$ be non-negative integers. Let $\P_M(k,m)$ be the set of partitions into exactly $M$ parts with the smallest part $\geq k$ and the gap between consecutive parts $\geq m$. We let $\P_0(k,m)$ be the set that contains the partition of zero, the empty sequence. The sets $\P_M(k,m)$ are mutually disjoint for distinct integers $M$ and satisfy the properties $\P_M(k,m+1)\subset \P_M(k,m)$ and $\P_M(k+1,m)\subset \P_M(k,m)$. For brevity we define the short-hand notation 
\begin{equation}\label{larger_than_or_equal_to_M_Sets}\P_{\leq M} (k,m) = \bigcup_{l=0}^M \P_l(k,m).
\end{equation}
Clearly we have $\U=\lim_{M\rightarrow\infty}\P_{\leq M}(1,0)$, $\D=\lim_{M\rightarrow\infty}\P_{\leq M}(1,1)$, and $\R\R_i = \lim_{M\rightarrow\infty}\P_{\leq M} (i,2)$ for $i\in\{1,2\}$.

\begin{theorem}\label{Finite_Weighted_GF} For a partition $\pi = (\lambda_1,\lambda_2,\dots)$, 
\begin{equation}
\label{GF_M_parts_general_distance_smallest_part}\sum_{\pi\in \P_M(k,m)} \omega_{k,m}(\pi) q^{|\pi|} = \frac{q^{m{M\choose2}+kM}}{(q)_M^2},
\end{equation}
where 
\begin{equation}
\label{General_Weight} \omega_{k,m}(\pi) :=(\lambda_{M}+1-k)\cdot\prod_{i=1}^{M-1} (\lambda_{i}-\lambda_{i+1}+1-m).
\end{equation} 
\end{theorem}

\begin{proof} The $M=0$ case is obvious. Let $M$ and $k$ be positive and $m$ be non-negative. A combinatorially interpretation of $q^{m{M\choose 2}}$ is that it is the generating function for the partition $\pi_1:=((M-1)m,(M-2)m,\dots,2m,m)$. The partition $\pi_1$ is into $M-1$ distinct parts when $m$ and $M-1$ are non-zero and it is the empty partition otherwise. The $q^{kM}$ term is interpreted as the partition $\pi_2$ into $M$ parts each equal to $k$. 

Point-wise addition of two partitions can be defined as putting the $i$-th rows of the Ferrers diagrams back to back. The empty partition is the identity element of the defined point-wise addition. This operation on partitions yields new partitions. Point-wise addition of $\pi_1$ and $\pi_2$ gives $\pi^* := ((M-1)m+k,(M-2)m+k,\dots,2m+k,m+k,k)$. 

The partition $\pi^*$ has smallest norm satisfying the properties of $P_M(k,m)$. We consider this smallest partition to be colorless. As a generating function, \begin{equation}\label{1/(q)^2_M}\frac{1}{(q)^2_M}\end{equation} keeps count of the partitions into $\leq M$ parts  where every column in the Ferrers diagram can come in one of two colors counted disregarding the order of these colors. Adding a partition $\pi'$ counted by \eqref{1/(q)^2_M} with $\pi^*$ point-wise we get a partition $\pi\ = (\lambda_1,\lambda_2,\dots, \lambda_{M-1},\lambda_M) \in \P_M(k,m)$.

There are $(\lambda_1-\lambda_2 + 1-m)$ many color combinations for the colored portion of the first part of $\pi$, $\lambda_1$, $(\lambda_2-\lambda_3+1-m)$ for colored portion of $\lambda_2$, and so on. The colored piece of $\lambda_M$ comes in $\lambda_M+1-k$ many possible color combinations disregarding the order of colors. Therefore, there is a total of \[(\lambda_{M}+1-k)\cdot\prod_{i=1}^{M-1} (\lambda_{i}-\lambda_{i+1}+1-m)\] possibilities for the partition $\pi\in\P_M(k,m)$ that are counted by the generating function on the right-hand side of \eqref{GF_M_parts_general_distance_smallest_part}.
\end{proof}

\begin{corollary}\label{Corollary_of_Finite_Weighted_GF} For $\label{general_weight}\omega_{k,m}(\pi) $ as in \eqref{General_Weight}
\begin{equation}\label{Corollary_of_Finite_Weighted_GF_EQN}
\sum_{\pi\in \P_{\leq M}(k,m)} \omega_{k,m}(\pi) q^{|\pi|} = \sum_{i=0}^{M} \frac{q^{m{i\choose2}+ki}}{(q)_i^2}.
\end{equation}
\end{corollary}

Letting $M\rightarrow\infty$  with $(k,m)=(1,2)$ in Corollary~\ref{Corollary_of_Finite_Weighted_GF} proves Theorem~\ref{Alladi_weighted_sum} by the first equality of \eqref{GF_unrestricted_pts}. We also note that Theorem~\ref{Finite_Weighted_GF} is the refinement of the finite analogue of Theorem~\ref{Alladi_weighted_sum}, \cite[Thorem~3]{AlladiWeighted}, which connects the Durfee square sizes and the number of parts of partitions of the first Rogers--Ramanujan type, $\R\R_1$. Letting $(k,m)=(1,2)$ in Theorem~\ref{Finite_Weighted_GF} proves \cite[Thorem~3]{AlladiWeighted}.

An interesting connection with the classical study of partitions comes from the choice of $(k,m)=(2,2)$ and letting $M\rightarrow\infty$ in Corollary~\ref{Corollary_of_Finite_Weighted_GF}. 

\section{Connections with the crank}\label{Section4}

In 1988, Andrews and Garvan  \cite{FrankCrank} found an explanation of the crank of an ordinary partition $\pi$. Explicitly, the crank of a partition is defined as \[cr(\pi):=\left\{ \begin{array}{cl}
\text{largest part of }\pi, &\text{if }1\text{ is not a part of }\pi,\\
\#\text{ of parts larger than }\#\text{ of }1\text{s} - \#\text{ of }1\text{s in }\pi, & \text{otherwise}.
\end{array}\right.\] Let $\C_{=M}$, $\C_{\leq M}$, and $\C_{\geq M}$ be the sets of partitions with crank $=M$, $\leq M$, and $\geq M$ respectively. We have

\begin{theorem} \label{RR2_THM}
\begin{equation}\label{omega_22}\sum_{\pi\in \R\R_2} \omega_{2,2}(\pi) q^{|\pi|} = \sum_{\pi \in \C_{\geq 0}} q^{|\pi|}.\end{equation}
\end{theorem}

Let $(k,m)=(2,2)$ and $M\rightarrow\infty$ in \eqref{Corollary_of_Finite_Weighted_GF_EQN}. Comparison between \cite[(3)]{Auluck} and \cite[(11)]{Auluck} with the use of the right-hand side equation of \eqref{GF_unrestricted_pts} shows \[\sum_{i\geq 0} \frac{q^{i^2+i}}{(q)_i^2} =  \frac{1}{(q)_\infty} \sum_{i\geq 0} (-1)^{i} q^{i+1\choose 2}. \] The right-hand side of the above line is the summation over $k\geq 0$ of the Dyson's equation for a fixed crank $k$, \cite[(3.1)]{BerkovichGarvanCrank}. This yields Theorem~\ref{RR2_THM}.

It is not clear that for $n\geq 2$ the number of partitions of $n$ with positive crank is the same as the number of partitions of $n$ with negative crank. A combinatorial proof of this phenomenon as well as refinements of the fixed crank's generating functions can be found in \cite{BerkovichGarvanCrank}. 

We would like to point out that for $k\geq 2$ the weights $\omega_{k,m}(\pi)$ take the identical values on the sets $\P_M(k-1,m)$ and $\P_M(k,m)$. Therefore, by taking the difference of $\omega_{1,2}$ and $\omega_{2,2}$ on the set $\R\R_1$ one can show

\begin{theorem}\label{Tilde_Weights_THEOREM} Let $\nu(\pi)$ denote the number of parts of $\pi$. Then
\begin{align}
\label{negative_crank}\sum_{\pi\in \R\R_1} \tilde{\omega}_{1}(\pi) q^{|\pi|} &= \sum_{\pi \in \C_{\leq -1}} q^{|\pi|} = q+\sum_{\pi \in \C_{\geq 1}} q^{|\pi|} ,\\
\label{zero_crank}\sum_{\pi\in \R\R_1} \tilde{\omega}_{2}(\pi) q^{|\pi|} &= -q + \sum_{\pi \in \C_{=0}} q^{|\pi|} , 
\end{align}
where
\begin{equation}\label{Tilde_Weights}\begin{array}{ccc}\displaystyle
\tilde{\omega}_{1}(\pi) = \prod_{i=1}^{\nu(\pi)-1} (\lambda_{i}-\lambda_{i+1}-1) &\text{and}& \displaystyle\tilde{\omega}_{2}(\pi) = (\lambda_{\nu(\pi)}-2)\cdot\prod_{i=1}^{\nu(\pi)-1} (\lambda_{i}-\lambda_{i+1}-1).
\end{array}\end{equation}
\end{theorem}

Replacing $\R\R_2$ with $\R\R_1$ on the left-hand side of \eqref{omega_22} and subtracting this from \eqref{omega_12} side-by-side proves \eqref{negative_crank}. The second equality of \eqref{negative_crank} is due to \cite{BerkovichGarvanCrank}. Difference of \eqref{omega_22} and \eqref{negative_crank} shows \eqref{zero_crank}.

Let $\D_l$ to be the subset of $\D$ that consists of the partitions into exactly $l$ distinct parts. Recall that $\mathcal{O}(\pi) := \lambda_1 + \lambda_3+\dots$, the sum of the odd indexed parts, for a partition $\pi=(\lambda_1,\lambda_2,\dots)$.

\begin{theorem}\label{Weight_Change_Odd_Indexed_Parts_with_number_of_parts_known} For $l$ a non-negative number and $v\in\{0,1\}$, we have
\begin{equation}\label{Odd_Index_Part_Weights}\sum_{\pi\in \D_{2l+v}} q^{\mathcal{O}(\pi)} = \sum_{\pi \in \P_{l+v}(2-v,2)} \left[(1-v)\omega_{2,2}(\pi)+v\tilde{\omega}_1(\pi)\right] q^{|\pi|}.\end{equation}
\end{theorem}

\begin{proof} Let $\pi = (\lambda_1,\lambda_2,\dots,\lambda_{2l+v})$ be a partition in $D_{2l+v}$. Consider the projection mapping $\textbf{P}_{2l+v}:\D_{2l+v} \rightarrow \P_{l+v}(2-v,2)$ as $\textbf{P}_{2l+v}(\pi) = (\lambda_1^*,\lambda_2^*,\dots,\lambda_{l+v}^*)=(\lambda_1,\lambda_3,\dots, \lambda_{2l+(-1)^{v+1}})$. Therefore, $\mathcal{O}(\pi) = |\textbf{P}_{2l+v}(\pi)|$. 

The number of pre-images of a partition must be counted for the verification of \eqref{Odd_Index_Part_Weights}. Given $\textbf{P}_{2l+v}(\pi)$, there are $(\lambda_1^*-\lambda_2^*-1)=(\lambda_1-\lambda_3-1)$ possible $\lambda_2$'s in the pre-image, $(\lambda_2^*-\lambda_3^*-1)=(\lambda_3-\lambda_5-1)$ possible $\lambda_4$'s in the pre-image, and so on. Hence, the total number of possible $\pi \in \D_{2l+v}$ that would project to $\textbf{P}_{2l+v}(\pi) \in \P_{l+v}(2-v,2)$ is \[\prod_{i=1}^{l+v-1} (\lambda_{2l-1}-\lambda_{2l+1}-1)\times \text{``the number of possibilities for the smallest part."}\]

Depending on $v$ the number of possibilities for the smallest part changes. If $v=1$, then $\lambda_{l+1}^*=\lambda_{2l+1}$ is the smallest part. There is only one possibility for the smallest part, which makes the weight of $\textbf{P}_{2l+1}(\pi)$ to be $\tilde{\omega}_1(\pi)$ for a $\pi\in D_{2l+1}$. If $v=0$, then $\lambda_{l}^* = \lambda_{2l-1}$ is the second smallest part of the pre-image $\pi$. Hence, there are $(\lambda_{2l-1}-1)$ possibilities for the non-zero smallest part $\lambda_{2l}$ of $\pi$, which shows that the weight of $\textbf{P}_{2l}(\pi)$ is $\omega_{2,2}(\pi)$ for a $\pi\in D_{2l}$.
\end{proof}
 Observe that Theorem~\ref{Weight_Change_Odd_Indexed_Parts_with_number_of_parts_known} connects the generating function for the number of partitions with the same sum of the odd-indexed parts with all three theorems: Theorem~\ref{Alladi_weighted_sum}, Theorem~\ref{RR2_THM} and Theorem~\ref{Tilde_Weights_THEOREM}. For any partition $\pi\in\R\R_1$, $w_{1,2}(\pi) =\omega_{2,2}(\pi)+\tilde{\omega}_1(\pi)$ by \eqref{General_Weight} and \eqref{Tilde_Weights}. Therefore, summing the left-hand side of \eqref{Odd_Index_Part_Weights} over all $l$ and $v$ yields \begin{equation}\label{proof_of_Ali_THM1}\sum_{\pi\in \D} q^{\mathcal{O}(\pi)}=\Sum_{l,v\geq 0}\ \sum_{\pi\in \D_{2l+v}} q^{\mathcal{O}(\pi)} = \sum_{\pi\in \R\R_1} \omega_{1,2}(\pi) q^{|\pi|}. \end{equation} This shows Theorem~\ref{AliTHM1} as a corollary of Theorem~\ref{Weight_Change_Odd_Indexed_Parts_with_number_of_parts_known} using Theorem~\ref{Alladi_weighted_sum}. Similarly letting $v=0$ in \eqref{Odd_Index_Part_Weights} and summing over non-negative $l$ gives \begin{equation}\label{Half_of_RR2_Crank_result}\sum_{\pi\in \D_e} q^{\mathcal{O}(\pi)}=\sum_{l\geq 0}\ \sum_{\pi\in \D_{2l}} q^{\mathcal{O}(\pi)} = \sum_{\pi\in \R\R_2} \omega_{2,2}(\pi) q^{|\pi|}, \end{equation} of Theorem~\ref{RR2_THM}, where $\D_e$ is the set of partitions into even number of distinct parts. Similarly $v=1$ gives the connection of this partition statistic with \eqref{negative_crank}. Combinatorially, the equation \eqref{Half_of_RR2_Crank_result} with Theorem~\ref{RR2_THM} gives:
 \begin{theorem}
The number of partitions into even number of distinct parts whose odd-indexed parts' sum is $n$ is equal to the number of partitions of $n$ with non-negative crank
 \end{theorem}
 
The proof of Theorem~\ref{Weight_Change_Odd_Indexed_Parts_with_number_of_parts_known} shows that replacing the norm with a partition statistic in the generating function (such as $\mathcal{O}(\pi)$) may be related with generating function for the weighted count of partitions. Changing statistics in itself is an interesting question. Moreover, as exemplified in \eqref{proof_of_Ali_THM1} and \eqref{Half_of_RR2_Crank_result}, the study of writing new generating functions with respect to different statistics instead of the norm would also yield non-trivial examples of \eqref{Constant_Weights_Def}.

\section{Generating Functions with respect to the Sum of Odd-indexed Parts}\label{Section5}

We want to remind the reader that the convenience of writing the generating functions for the number of partitions with respect to the norm comes from there being a finite number of partitions having the prescribed norm. The same applies for $\mathcal{O}(\pi)$. There are finite number of partitions $\pi$ with $\mathcal{O}(\pi)=n$. This is not necessarily true for all the partition statistics. An analogue of the statistics $\mathcal{O}(\pi)$ is an example of this observation. Let $\E(\pi):=\lambda_2+\lambda_4+\dots$, the sum of all the even-indexed parts, for $\pi= (\lambda_1,\lambda_2,\dots)$, then all partitions $\pi_i = (\lambda_2+i,\lambda_2,\dots)$ for $i\in \mathbb{Z}_{{>0}}$ satisfy $\E(\pi_i)=\E(\pi_j)$, $\forall i,\ j\in \mathbb{Z}_{{>0}}$. 

Writing generating functions with some natural partition statistics such as $\mathcal{O}$ and $\E$ can be studied directly from the the results of \cite{BerkovichUncu2} ,\cite{Boulet}, and \cite{Masao}. One decorates the Ferrers diagrams by writing on the dots on the rows. On odd-indexed rows one puts alternating $a$ and $b$'s starting from an $a$ and does the same with $c$ and $d$ for the even-indexed rows starting from a $c$. We call these \textit{four-decorated} Ferrers diagrams. In 2006, Boulet \cite{Boulet} found explicit formulas for the generating functions for four-decorated Ferrers diagrams from the sets $\U$ and $\D$. Let $\Phi(a,b,c,d)$ be the generating function for the weighted four-decorated Ferrers diagrams and $\Psi(a,b,c,d)$ be the generating function for the weighted four-decorated Ferrers diagrams that has distinct row sizes. Abstractly
\[
\Psi(a,b,c,d) := \sum_{\pi\in \D} \omega_\pi(a,b,c,d),\text{ and }\ \Phi(a,b,c,d) := \sum_{\pi\in \U} \omega_\pi(a,b,c,d),
\]
where \[\omega_\pi (a,b,c,d):=a^{\#\text{ of }a's}b^{\#\text{ of }b's}c^{\#\text{ of } c's}d^{\#\text{ of } d's}.\]

Explicit formulas of these generating functions are given by the following theorem.

\begin{theorem}[Boulet, 2006] \label{Boulet_THM}For variables $a$, $b$, $c$, and $d$ and $Q:=abcd$, we have\[\Psi(a,b,c,d) :=\frac{(-a,-abc;Q)_\infty}{(ab;Q)_\infty},\text{  and  }
\Phi(a,b,c,d) := \frac{(-a,-abc;Q)_\infty}{(ab,ac,Q;Q)_\infty}.\]\end{theorem}

Theorem~\ref{Boulet_THM} provides another direct proof of Theorem~\ref{AliTHM1}: $\Psi(q,q,1,1,)$. Similarly $\Phi(q,q,1,1)$, and $\Psi(q,1,q,1)$ yield

\begin{theorem}\label{AliTHM2}
\[\sum_{\pi\in \U} q^{\mathcal{O}(\pi)} = \frac{1}{(q)_\infty^2},\ \text{  and  }\ 
\label{corollary_part_2}\sum_{\pi\in \D} q^{\mathcal{O}(\pi')} = (-q)^2_\infty,\]
where $\pi'$ is the conjugate of the partition $\pi$.
\end{theorem}

Theorem~\ref{AliTHM2} can easily be translated to combinatorial results using vector partitions. Some non-trivial examples of \eqref{Constant_Weights_Def} coming from Theorem~\ref{AliTHM2} similar to \eqref{proof_of_Ali_THM1} and \eqref{Half_of_RR2_Crank_result} are as follows:

\begin{corollary}\label{COR_ALI2} Let $par(n)$ be the parity of $n$, $par(n)=0$ if $n$ even and $1$ otherwise. Let $\K$ be the set of partitions with gaps between parts and the smallest part both $\leq 2$. Let $\pi = (\lambda_1,\lambda_2,\dots)$, then
\[\sum_{\pi\in \U} q^{\mathcal{O}(\pi)} =\sum_{\pi\in \U}\omega_{0,0}(\pi) q^{|\pi|} ,\ \text{  and  }\
\sum_{\pi\in \D} q^{\mathcal{O}(\pi')} =\sum_{\pi\in \K} \hat{\omega}_1(\pi) q^{|\pi|},\] where \[\omega_{0,0} (\pi) := (\lambda_{\nu(\pi)}+1)\cdot \prod_{i=1}^{\nu(\pi)-1} (\lambda_i-\lambda_{i+1}+1),\ \text{ and }\ \hat{\omega}_1(\pi) := 2^{par(\lambda_{\nu(\pi)})}\cdot \prod_{i=1}^{\nu(\pi)-1} 2^{par(\lambda_i - \lambda_{i+1})}.\]
\end{corollary}

Mark that $\omega_{0,0}(\pi)$ fits the definition \eqref{General_Weight}. For uniform definition of a partition and the definition of $\P_{M}(k,m)$ sake, we have ignored this case in the prior conversation. We can also replace $\mathcal{O}(\pi')$ with $\E(\pi')$ on the set $\D$. This change would introduce a factor of 2 on the right-hand side of the respective identities of Theorem~\ref{AliTHM2} and Corollary~\ref{COR_ALI2}.

\section{Conclusion}\label{Section7}
%

It appears that the approach of writing generating functions with statistics other than the norm offers a wide variety of questions. One of these questions is identifying a statistic and a set of partitions for given weighted count. There are many examples of Theorem~\ref{Alladi_weighted_sum} like weighted partition identities. Having a way of connecting these type of identities with the generating functions written with respect to a partition statistics would expand the horizons of this study.

The crank's appearance in this study is fortunate but not unexpected. The author would like to recall that Theorem~\ref{RR2_THM} can been presented as a non-trivial example of \eqref{Constant_Weights_Def}.
\begin{theorem}
\[ \sum_{\pi\in {\D}_{e}} q^{\mathcal{O}(\pi)} = \sum_{\pi \in \C_{\geq 0}} q^{|\pi|},\]
where $\widehat{\D}_{e}$ is the set of partitions into even number of distinct parts.
\end{theorem} 
In a similar fashion, Theorem~\ref{Tilde_Weights_THEOREM} can also be represented in the partition statistic $\mathcal{O}$ as
\begin{theorem} Let $\D_o$ denote the set of partitions into odd number of distinct parts. Then,
\begin{align*}
\sum_{\pi\in {\D}_{o}} q^{\mathcal{O}(\pi)} &= \sum_{\pi \in \C_{\leq -1}} q^{|\pi|},
\intertext{and}
\sum_{\pi\in {\D}} (-1)^{\nu(\pi)}q^{\mathcal{O}(\pi)} &= -q+\sum_{\pi \in \C_{= 0}} q^{|\pi|},
\end{align*}
where $\nu(\pi)$ represents the number of parts of the partition $\pi$.
\end{theorem}

We note that \cite{BerkovichUncu2}, and \cite{Masao} makes it possible to refine the results involving the partition statistic $\mathcal{O}$ by imposing bounds on the number of parts and the largest part of partitions. Moreover, there are many more fundamental statistics of partitions similar to $\mathcal{O}$, and $\E$. It would be of interest to see results and weights related to the rank of a partition and similar known classical partition statistics. The author is planning on addressing these observations in the future.

\section{Acknowledgement}

The author would like to thank George E. Andrews and Alexander Berkovich for their guidance. The author would also like to thank Alexander Berkovich, Jeramiah Hocutt, Frank Patane, and John Pfeilsticker for their helpful comments on the manuscript.

\end{document}